\batchmode
\makeatletter
\def\input@path{{\string"/Users/russw/Documents/Research/mypapers/Divisibility of binomial coefficients/\string"}}
\makeatother
\documentclass[12pt,oneside,english]{amsart}
\usepackage[T1]{fontenc}
\usepackage[latin9]{inputenc}
\usepackage{geometry}
\usepackage{babel}
\usepackage{url}
\usepackage{enumitem}
\usepackage{amsthm}
\usepackage{amssymb}
\usepackage{graphicx}
\usepackage[unicode=true,pdfusetitle,
 bookmarks=true,bookmarksnumbered=false,bookmarksopen=false,
 breaklinks=false,pdfborder={0 0 1},backref=false,colorlinks=false]
 {hyperref}

\makeatletter

\providecommand{\tabularnewline}{\\}

\numberwithin{equation}{section}
\numberwithin{figure}{section}
\theoremstyle{plain}
\newtheorem{thm}{\protect\theoremname}[section]
  \theoremstyle{plain}
  \newtheorem{question}[thm]{\protect\questionname}
  \theoremstyle{plain}
  \newtheorem{prop}[thm]{\protect\propositionname}
  \theoremstyle{plain}
  \newtheorem{lem}[thm]{\protect\lemmaname}
  \theoremstyle{remark}
  \newtheorem{rem}[thm]{\protect\remarkname}
  \theoremstyle{plain}
  \newtheorem{cor}[thm]{\protect\corollaryname}

\usepackage{ae}
\usepackage{aecompl}

\makeatletter
\newcommand{\setlabel}[1]{\def\@currentlabel{#1}}
\makeatother

\makeatother

  \providecommand{\corollaryname}{Corollary}
  \providecommand{\lemmaname}{Lemma}
  \providecommand{\propositionname}{Proposition}
  \providecommand{\questionname}{Question}
  \providecommand{\remarkname}{Remark}
\providecommand{\theoremname}{Theorem}

\begin{document}
\global\long\def\cosetlat{\mathcal{C}}

\global\long\def\sglat{\mathcal{L}}

\global\long\def\ff{\mathbb{F}}

\global\long\def\nn{\mathbb{N}}

\global\long\def\zz{\mathbb{Z}}

\global\long\def\normalin{\mathrel{\triangleleft}}

\global\long\def\semidirect{\rtimes}

\global\long\def\comma{{,}}

\global\long\def\join{\mathbin{\ast}}

\title[Divisibility of binomials and generation of alternating groups]{Divisibility of binomial coefficients and generation of alternating
groups}

\author{John Shareshian and Russ Woodroofe}

\thanks{The first author was supported in part by NSF Grants DMS-0902142
and DMS-1202337.}

\address{Department of Mathematics, Washington University in St.~Louis, St.~Louis,
MO, 63130}

\email{shareshi@math.wustl.edu}

\address{Department of Mathematics \& Statistics, Mississippi State University,
Starkville, MS 39762}

\curraddr{University of Primorska, UP FAMNIT, Glagoljaška 8, 6000 Koper, Slovenia }

\email{rsw9@cornell.edu}

\urladdr{\url{http://rwoodroofe.math.msstate.edu/}}
\begin{abstract}
We examine an elementary problem on prime divisibility of binomial
coefficients. Our problem is motivated by several related questions
on alternating groups. 
\end{abstract}

\maketitle

\section{\label{sec:Introduction}Introduction}

We will discuss several closely related problems. The first is an
elementary problem concerning divisibility of binomial coefficients
by primes. Consider the following condition that a positive integer
$n$ might satisfy:
\begin{enumerate}
\item \label{enu:BinomCond}There exist primes $p$ and $r$ such that if
$1\leq k\leq n-1$, then the binomial coefficient ${n \choose k}$
is divisible by at least one of $p$ or $r$.
\end{enumerate}
\begin{question}
\label{que:MainQBinom} Does Condition (\ref{enu:BinomCond}) hold
for all positive integers $n$?
\end{question}

We were led to ask Question~\ref{que:MainQBinom} by a problem on
the alternating groups. Indeed, we consider several related group-theoretic
conditions on a positive integer $n$:
\begin{enumerate}[resume]
\item \label{enu:UnivGenCond}There exist primes $p$ and $r$ such that
if $H<A_{n}$ is a proper subgroup, then the index $[A_{n}:H]$ is
divisible by at least one of $p$ or $r$.
\item [(2')]\setcounter{enumi}{2}There exist primes $p$ and $r$ such
that if $P$ is a Sylow $p$-subgroup and $R$ a Sylow $r$-subgroup
of $A_{n}$, then $\left\langle P,R\right\rangle =A_{n}$.
\item \label{enu:PPOverSylConjGenCond}There exist a prime $p$ and a conjugacy
class $D$ in $A_{n}$ consisting of elements of prime power order,
such that if $P$ is a Sylow $p$-subgroup of $A_{n}$ and $d\in D$,
then $\left\langle P,d\right\rangle =A_{n}$.
\item \label{enu:PPConjGenCond}There exist conjugacy classes $C$ and $D$
in $A_{n}$, both consisting of elements of prime power order, such
that if $(c,d)\in C\times D$, then $\left\langle c,d\right\rangle =A_{n}$.
\item \label{enu:PrimeConjGenCond}\label{enu:LastCond}There exist conjugacy
classes $C$ and $D$ in $A_{n}$, both consisting of elements of
prime order, such that if $(c,d)\in C\times D$, then $\left\langle c,d\right\rangle =A_{n}$. 
\end{enumerate}
If we wish to specify one or both of the primes, then we may say that
$n$ satisfies Condition~(1) with $p$, or that $n$ satisfies Condition~(1)
with $p$ and $r$. We'll use similar language for the other conditions. 

Conditions~(\ref{enu:UnivGenCond}) and (\ref{enu:UnivGenCond}')
are equivalent, and each condition in the above list implies the previous
condition. That is, for any positive integer $n$ the following chain
of implications holds, where the primes $p$ and $r$ may be held
fixed. 
\begin{equation}
(5)\implies(4)\implies(3)\implies(2')\iff(2)\implies(1).\label{eq:ConditionsImplications}
\end{equation}
See also Proposition~\ref{thm:EquivUnivGenBinomDiv} below.

All implications in (\ref{eq:ConditionsImplications}) are completely
trivial or immediate from the definition of a Sylow subgroup, with
the exception of the implication $(2)\implies(1)$. This implication
follows since $A_{n}$ has subgroups of index ${n \choose k}$ for
each $0\leq k\leq n$. (The stabilizer in $A_{n}$ of a $k$-subset
of $[n]$ is such a subgroup.)

There are infinitely many positive integers $n$ that do not satisfy
Condition~(\ref{enu:PrimeConjGenCond}). However, the set of such
integers is rather sparse, and likely very sparse. See Proposition
\ref{prop:PrimeGenerationFails} and Theorem \ref{thm:AsymptoticDensity}
below. We are not aware of any integer $n$ for which Conditions~(\ref{enu:BinomCond})\textendash (\ref{enu:PPConjGenCond})
fail to hold. In addition to Question~\ref{que:MainQBinom}, we will
consider the following.

\theoremstyle{plain}\newtheorem*{questions}{Questions \ref*{sec:Introduction}.\ref*{enu:UnivGenCond}--\ref*{sec:Introduction}.\ref*{enu:PPConjGenCond}}\begin{questions}

\phantomsection

\setlabel{\ref*{sec:Introduction}.\ref*{enu:UnivGenCond}}\label{que:MainQUnivGen}

\setlabel{\ref*{sec:Introduction}.\ref*{enu:PPOverSylConjGenCond}}\label{que:MainQPPOverSyl}

\setlabel{\ref*{sec:Introduction}.\ref*{enu:PPConjGenCond}}\label{que:MainQPPGen}

Do Conditions (\ref{enu:UnivGenCond})\textendash (\ref{enu:PPConjGenCond})
hold for all positive integers $n$?

\end{questions}

\subsection{Motivations and related questions}

Question~\ref{que:MainQBinom} fits into a line of inquiry going
back to \cite{Kummer:1852} on the distribution of binomial coefficients
that are divisible by a given prime. The remaining conditions and
questions arose from our work and that of others on generation of
finite simple groups. Recall that the Classification of Finite Simple
Groups tells us that every simple group is isomorphic to one of the
following: an alternating group $A_{n}$ with $n\geq5$, a cyclic
group of prime order, a group of Lie type, or one of twenty six sporadic
groups. Conditions analogous to Conditions~(\ref{enu:UnivGenCond})\textendash (\ref{enu:LastCond})
are known or conjectured for sporadic and Lie type groups.

We ourselves became interested in these problems via Question~\ref{que:MainQUnivGen}.
In \cite{Shareshian/Woodroofe:2016}, we define a group $G$ to be
\emph{universally $(p,r)$-generated} if $G=\left\langle P,R\right\rangle $
for any Sylow $p$-subgroup $P$ and Sylow $r$-subgroup $R$. (Compare
with Condition~(\ref{enu:UnivGenCond}')!) We say $G$ is \emph{universally
$(2,*)$-generated }if there is some prime $p$ such that $G$ is
universally $(2,p)$-generated.\emph{ }We showed the following.
\begin{thm}[Shareshian and Woodroofe \cite{Shareshian/Woodroofe:2016}]
 \label{thm:NonaltGroupsUnivgen}If $G$ is a finite simple group
that is abelian, of Lie type, or sporadic, then $G$ is universally
$(2,*)$-generated.
\end{thm}

We used Theorem~\ref{thm:NonaltGroupsUnivgen}, along with fixed-point
theorems of Smith \cite{Smith:1941} and of Oliver \cite{Oliver:1975},
to show that the order complex of the coset poset of any finite group
is non-contractible. 

In light of Theorem~\ref{thm:NonaltGroupsUnivgen}, it is natural
to ask whether $A_{n}$ is universally $(2,*)$-generated for every
$n$ \textendash{} that is, whether every $n$ satisfies Condition~(\ref{enu:UnivGenCond})
with $2$. This is not the case. The first failure of universal $(2,*)$-generation
is at $n=7$. It may be easier to understand the second failure, at
$15$, since $n=15$ does not even satisfy Condition~(\ref{enu:BinomCond})
with $2$. Question~\ref{que:MainQUnivGen} naturally suggests itself.
We will further discuss the case $p=2$ below in Section~\ref{subsec:p=00003D2}. 

We found that similar conditions had been examined earlier. The general
problem of generation by elements selected from fixed conjugacy classes
has been more broadly studied under the name of ``invariable generation''.
See for example \cite{Detomi/Lucchini:2015,Dixon:1992,Eberhard/Ford/Green:2015UNP,Kantor/Lubotzky/Shalev:2011}.
Dolfi, Guralnick, Herzog and Praeger first ask Question~\ref{que:MainQPPGen}
in \cite[Section 6]{Dolfi/Guralnick/Herzog/Praeger:2012}. These authors
conjecture that the analogue of Condition~(\ref{enu:PrimeConjGenCond})
holds for all but finitely many simple groups of Lie type, but point
out that the corresponding statement for alternating groups occasionally
fails.

Condition~(\ref{enu:PPOverSylConjGenCond}) interpolates naturally
between Conditions~(\ref{enu:UnivGenCond}) and (\ref{enu:PPConjGenCond}).
Although they do not ask Question~\ref{que:MainQPPOverSyl}, Damian
and Lucchini show in \cite{Damian/Lucchini:2007} that an analogue
of Condition~(\ref{enu:PPOverSylConjGenCond}) holds for many sporadic
simple groups and groups of Lie type. Indeed, they show that many
simple groups are generated by a Sylow $2$-subgroup $P$ together
with any element of a certain conjugacy class consisting of elements
of prime order. 

\subsection{Results for arbitrary primes}

Our first result adds an additional implication to the list in (\ref{eq:ConditionsImplications}).
\begin{thm}
\label{thm:EquivUnivGenBinomDiv} Let $p$ and $r$ be primes. If
the positive integer $n$ is not a prime power, then Conditions~(\ref{enu:BinomCond})
and (\ref{enu:UnivGenCond}) are equivalent for $n$ with $p$ and
$r$.
\end{thm}

The case where $n$ is a prime power is not difficult. 
\begin{prop}
\label{pro:UgenAltPP} If $n$ is a power of the prime $p$, then

\begin{enumerate}
\item[(A)] \setlabel{A} $n$ satisfies Condition~(\ref{enu:PPOverSylConjGenCond})
with a Sylow $2$-subgroup unless $n=7$, and
\item[(B)] \setlabel{B} $n$ satisfies Condition~(\ref{enu:PPConjGenCond})
with $p$.
\end{enumerate}
\end{prop}

In particular, it follows from Theorem~\ref{thm:EquivUnivGenBinomDiv}
and Proposition~\ref{pro:UgenAltPP} that Questions~\ref{que:MainQBinom}
and \ref{que:MainQUnivGen} are equivalent. We remark that the requirement
that $n\neq7$ in Proposition~\ref{pro:UgenAltPP}~(A) is necessary,
as $n=7$ satisfies Condition~(\ref{enu:BinomCond}), but not Condition~(\ref{enu:UnivGenCond}),
with the prime $2$.

While Questions~\ref{que:MainQBinom}\textendash \ref{que:MainQPPGen}
are still open, we have amassed a large collection of integers for
which the answers are ``yes''. The \emph{asymptotic density} \cite{Niven/Zuckerman/Montgomery:1991}
of a set $S$ of positive integers is defined to be 
\[
\liminf_{M\rightarrow\infty}\frac{\left|S\cap[M]\right|}{M}.
\]

Dolfi, Guralnick, Herzog and Praeger remark in \cite{Dolfi/Guralnick/Herzog/Praeger:2012}
that Condition~(\ref{enu:PrimeConjGenCond}) appears likely to hold
with asymptotic density 1. We show the following.
\begin{thm}
\label{thm:AsymptoticDensity}Let $\alpha$ be the asymptotic density
of the set of positive integers $n$ that satisfy Condition~(\ref{enu:PrimeConjGenCond}),
and let $\rho$ denote the Dickman-de Bruijn function (see for example
\cite{Granville:2008}). We have 

\begin{enumerate}
\item [(A)] $\alpha\geq1-\rho(20)>1-10^{-28}$, and
\item [(B)] if either the Riemann hypothesis or the Cramér Conjecture holds,
then $\alpha=1$.
\end{enumerate}
\end{thm}

The authors also claim in \cite{Dolfi/Guralnick/Herzog/Praeger:2012}
that Condition~(\ref{enu:PrimeConjGenCond}) fails for infinitely
many values of $n$, and that the smallest $n$ for which Condition~(\ref{enu:PrimeConjGenCond})
fails is $210$. We will see that the first claim is true, but the
second is not.
\begin{prop}
\label{prop:PrimeGenerationFails}For any $a\geq3$, the integer $n=2^{a}$
fails to satisfy Condition~(\ref{enu:PrimeConjGenCond}). 
\end{prop}

Theorem~\ref{thm:AsymptoticDensity} suggests a positive answer to
Questions~\ref{que:MainQBinom}\textendash \ref{que:MainQPPGen}
for all but a vanishingly sparse set of large integers. We have also
examined many small integers with the aid of a computer, verifying
 the following.
\begin{prop}
\label{prop:MainComputation}Every $n\leq1\comma000\comma000\comma000$
satisfies Condition~(\ref{enu:UnivGenCond}).
\end{prop}

The key tool in the proofs of both Theorem~\ref{thm:AsymptoticDensity}
and Proposition~\ref{prop:MainComputation} is the following sieve
lemma.
\begin{lem}[Sieve Lemma]
\emph{ \label{lem:SgdivSieve}} Let $n\geq9$ be an integer. Let
$p$ and $r$ be primes, and let $a$ and $b$ be positive integers.

\begin{enumerate}
\item[(A)] \setlabel{A}\label{enu:SgSieveUgen}If $n$ is not a prime power,
$p^{a}$ divides $n$, and $r^{b}<n<r^{b}+p^{a}$, then $n$ satisfies
Condition~(\ref{enu:UnivGenCond}) with $p$ and $r$.
\item[(B)] \setlabel{B}\label{enu:SgSievePrimeGen}If $p$ divides $n$ and
$r+2<n<r+p$, then $n$ satisfies Condition~(\ref{enu:PrimeConjGenCond})
with $p$ and $r$.
\end{enumerate}
\end{lem}

Theorem~\ref{thm:AsymptoticDensity} follows from combining Lemma~\ref{lem:SgdivSieve}
with known results on prime gaps and smooth numbers. We also use Lemma~\ref{lem:SgdivSieve}
to do much of the work in verifying Proposition~\ref{prop:MainComputation}. 

For those integers not handled by Lemma~\ref{lem:SgdivSieve}~(A),
Theorem~\ref{thm:EquivUnivGenBinomDiv} tells us that it suffices
to check divisibility of binomial coefficients. In particular, we
can avoid making any computations in large alternating groups. We
do not know how to avoid such computations for Condition~(\ref{enu:PPConjGenCond}).
The slow speed of these computations is the main obstacle to a computational
verification of Condition~(\ref{enu:PPConjGenCond}) for those values
of $n$ not addressed by Lemma~\ref{lem:SgdivSieve}.

\subsection{\label{subsec:p=00003D2}Results for $p=2$}

We return now to the case where one of the primes in Condition~(\ref{enu:UnivGenCond})
is $2$. Theorem~\ref{thm:NonaltGroupsUnivgen} suggests this case
as being particularly worthy of attention, and Proposition~\ref{pro:UgenAltPP}
gives infinitely many values of $n$ for which Condition~(\ref{enu:UnivGenCond})
holds with $2$.

However, there are also infinitely many positive integers $n$ that
do not even satisfy Condition~(\ref{enu:BinomCond}) with $2$. By
a theorem of Kummer (see Lemma~\ref{lem:KummerTheorem} below), if
$n=2^{a}-1$ for some positive integer $a$, then $\binom{n}{k}$
is odd for all $1\leq k\leq n-1$. (Indeed, a similar statement holds
for any prime $p$. In the language of group-actions, this says that
any Sylow $p$-subgroup of $S_{p^{a}-1}$ stabilizes a set of every
possible size $k$ with $1<k<p^{a}-1$.) Kummer's Theorem also implies
that there is no prime dividing every nontrivial $\binom{n}{k}$ unless
$n$ is a prime power. There are infinitely many $n$ of the form
$2^{a}-1$ that are not prime powers.

Using techniques similar to those for Proposition~\ref{prop:MainComputation},
we computationally verify the following.
\begin{prop}
\label{prop:2Computation}About 86.7\% of the positive integers $n\leq1\comma000\comma000$
satisfy Condition~(\ref{enu:UnivGenCond}) with $2$.
\end{prop}

\subsection{Organization}

We begin in Section~\ref{sec:Preliminaries} by giving necessary
background on maximal subgroups of alternating groups. In Section~\ref{sec:Kummers}
we state the well-known theorem of Kummer on prime divisibility of
binomial coefficients, and prove a analogue on prime divisibility
of the number of equipartitions of a set. We use these results in
Section~\ref{sec:Sieves} to prove Theorem~\ref{thm:EquivUnivGenBinomDiv},
Propositions~\ref{pro:UgenAltPP} and \ref{prop:PrimeGenerationFails},
and Lemma \ref{lem:SgdivSieve}. We also verify that Condition~(\ref{enu:PPConjGenCond})
holds for all small alternating groups. We apply Lemma~\ref{lem:SgdivSieve}
to prove Theorem~\ref{thm:AsymptoticDensity} in Section~\ref{sec:Asymptotic-density}.
We describe our computational verification of Propositions~\ref{prop:MainComputation}
and \ref{prop:2Computation} in Section~\ref{sec:Computation}.

\section*{Acknowledgements}

We thank Andrew Granville and Bob Guralnick for their thoughtful remarks.
A comment by Ben Green led to a significant improvement in the bound
given in Theorem~\ref{thm:AsymptoticDensity}~(A).

\section{\label{sec:Preliminaries}Preliminaries}

In this section we discuss necessary background on alternating and
symmetric groups. Readers familiar with basic facts about permutation
groups can safely skip this section.

In order to show that the index of every subgroup of the alternating
group $A_{n}$ is divisible by either $p$ or $r$, it suffices to
show the same for every maximal subgroup. The maximal subgroups of
$A_{n}$ are well-understood, as we now review. Additional background
can be found in \cite{Dixon/Mortimer:1996}, see also \cite{Liebeck/Praeger/Saxl:1987}.

We say that a subgroup $H\leq A_{n}$ is transitive or\emph{ }primitive
if the action of $H$ on $[n]$ satisfies the same property. That
is, $H$ is \emph{transitive} if for every $i,j\in[n]$, there is
some $\sigma\in H$ such that $i\cdot\sigma=j$. A transitive subgroup
$H$ is \emph{imprimitive} if there is a proper partition $\pi$ of
$[n]$ into sets of size greater than one, such that the parts of
$\pi$ are permuted by the action of $H$. If $H$ is transitive and
not imprimitive, then it is \emph{primitive.} Clearly, every subgroup
is either intransitive, imprimitive, or primitive. We examine maximal
subgroups of $A_{n}$ according to this trichotomy.

An intransitive subgroup $H$ is maximal in the (sub)poset of intransitive
subgroups of $A_{n}$ if and only if $H$ is the stabilizer in $A_{n}$
of some nonempty proper subset $X\subset[n]$. As $A_{n}$ sits naturally
in $S_{n}$, it is illuminating to also consider the stabilizer $H^{+}$
in $S_{n}$ of $X$. Then $H=H^{+}\cap A_{n}$. It is clear that $H^{+}\cong S_{\left|X\right|}\times S_{n-\left|X\right|}$.
If $\left|X\right|=k$, then it follows either from this isomorphism
or the Orbit-Stabilizer Theorem that 
\[
[A_{n}:H]=[S_{n}:H^{+}]=\frac{n!}{k!\cdot(n-k)!}={n \choose k}.
\]

Every imprimitive subgroup of $A_{n}$ stabilizes a partition of $[n]$.
It follows easily that a subgroup $H$ is maximal in the (sub)poset
of imprimitive subgroups of $A_{n}$ if and only if $H$ is the stabilizer
of a partition of $[n]$ into $n/d$ parts of size $d$ for some nontrivial
proper divisor $d$ of $n$. As in the intransitive case, we also
consider the stabilizer $H^{+}$ of the same partition in the action
by $S_{n}$. Then $H^{+}$ is isomorphic to the wreath product $S_{d}\wr S_{n/d}$.
Since $H=H^{+}\cap A_{n}$ (and $H^{+}\not\leq A_{n}$), we see that
\[
[A_{n}:H]=[S_{n}:H^{+}]=\frac{n!}{(d!)^{n/d}\cdot(n/d)!}.
\]
 By either the Orbit-Stabilizer Theorem or an elementary counting
argument, $[A_{n}:H]$ counts the number of partitions of $[n]$ into
$n/d$ equal-sized parts.

The index of a primitive proper subgroup of $A_{n}$ is typically
divisible by every prime smaller than $n$. See Theorem~\ref{thm:JordanPrimitiveA_n}
and the discussion following for a precise statement. 

\section{\label{sec:Kummers}Kummer's Theorem and an analogue}

\subsection{Kummer's Theorem}

We make considerable use of the following result of Kummer \cite{Kummer:1852}.
The most useful case of the lemma for us will be that where $a=1$.
See also \cite{Granville:1997} for an overview of related results.
\begin{lem}[{Kummer's Theorem \cite[pp115--116]{Kummer:1852}}]
\label{lem:KummerTheorem} Let $k$ and $n$ be integers with $0\leq k\leq n$.
If $a$ is a positive integer, then $p^{a}$ divides ${n \choose k}$
if and only if at least $a$ carries are needed when adding $k$ and
$n-k$ in base $p$.
\end{lem}

\subsection{An analogue for the number of equipartitions}

Lemma~\ref{lem:KummerTheorem} completely describes the prime divisibility
of indices of intransitive maximal subgroups of $A_{n}$. Lemma~\ref{lem:ImprimitiveDivIndex}
below provides a weaker but similarly useful characterization regarding
indices of imprimitive subgroups. Throughout this section, if $d$
is a nontrivial proper divisor of the positive integer $n$, then
we will write $I_{n,d}$ for the number of equipartitions of $n$
into parts of size $d$. Thus, 
\[
I_{n,d}=\frac{n!}{(d!)^{n/d}\cdot(n/d)!}.
\]
\begin{lem}
\label{lem:ImprimitiveDivIndex}Let $n$ be a positive integer, $d$
be a nontrivial proper divisor of $n$, and $p$ be a prime. Then
$p$ divides $I_{n,d}$ if and only if 

\begin{enumerate}
\item at least one carry is necessary when adding $n/d$ copies of $d$
in base $p$, and
\item $d$ is not a power of $p$.
\end{enumerate}
\end{lem}

\begin{proof}
It is straightforward to show by elementary arguments that 
\begin{equation}
I_{n,d}=\frac{1}{(n/d)!}\cdot\prod_{j=1}^{n/d}{jd \choose d}=\prod_{j=1}^{n/d}\frac{1}{j}{jd \choose d}=\prod_{j=1}^{n/d}{jd-1 \choose d-1}.\label{eq:WreathIndexProd1}
\end{equation}
 Our strategy is to use Lemma~\ref{lem:KummerTheorem} to examine
divisibility of the terms in these products.\medskip{}

\noindent \emph{Case 1.} $n/d<p$\smallskip{}

In this case $p$ does not divide $(n/d)!$, and the first condition
of the hypothesis implies the second. From (\ref{eq:WreathIndexProd1})
we thus see that $p$ divides $I_{n,d}$ if and only if $p$ divides
${jd \choose d}$ for some $1\leq j\leq n/d$. The claim for this
case then follows from Lemma~\ref{lem:KummerTheorem}. \medskip{}

\noindent \emph{Case 2.} $n/d\geq p$\smallskip{}

In this case a carry is always necessary when adding $n/d$ copies
of $d$, so we need only consider the second condition of the hypothesis. 

If the base $p$ expansion of $d$ has at least 2 nonzero places,
then there are at least 2 carries when adding $d$ to $pd-d$, as
the base $p$ expansion of $pd$ is obtained by shifting that of $d$
to the left by one place. It follows that $p^{2}$ divides ${pd \choose d}$,
hence that $p$ divides ${pd-1 \choose d-1}=\frac{1}{p}{pd \choose d}$.
By (\ref{eq:WreathIndexProd1}), $p$ divides $I_{n,d}$.

Otherwise, we have $d=kp^{a}$ for some $1\leq k<p$. Then the base
$p$ expansion of $(jd-1)-(d-1)=(j-1)kp^{a}$ vanishes below the $a$th
place. Also, the base $p$ expansion of $d-1$ is $(k-1)p^{a}+\sum_{i=0}^{a-1}(p-1)p^{i}$.
As the latter vanishes above the $a$th place, this place is the only
possible location for a carry in adding $d-1$ and $jd-1$. If $k=1$,
then the $a$th place of $d-1$ is $0$, so no carry occurs and for
no $j$ does $p$ divide ${jd-1 \choose d-1}$. If $k>1$, then a
carry occurs at the $a$th place for values of $j$ such that $(j-1)\cdot k\equiv p-1\mod p$.
(Such a $j<n/d$ exists since $\mathbb{Z}/p\mathbb{Z}$ is a field.)
\end{proof}
\begin{rem}
After submission of the paper, we became aware that a slightly different
(from Lemma~\ref{lem:ImprimitiveDivIndex}) characterization of prime
divisibility of $I_{n,d}$ appears as \cite[Lemma 2]{Thompson:1966}.
\end{rem}

\begin{cor}
\label{cor:ImprimitiveDivLargepp}Let $n$ and $b$ be positive integers
and $r$ be a prime, such that $n/2<r^{b}\leq n$. If $d$ is a nontrivial
proper divisor of $n$ which is not a power of $r$, then $r$ divides
$I_{n,d}$.
\end{cor}

\begin{proof}
Since $r^{b}>n/2$, there is a 1 in the $b$th place of the base $r$
expansion of $n$. On the other hand, $d\leq n/2$. Hence, the base
$r$ expansion of $d$ has a $0$ in the $b$th place. It follows
that there is at least one carry when we sum $n/d$ copies of $d$.
Lemma~\ref{lem:ImprimitiveDivIndex} then gives that $r$ divides
$I_{n,d}$ unless $d$ is a power of $r$.
\end{proof}
One can indeed extract from (\ref{eq:WreathIndexProd1}) the highest
power of $p$ dividing $I_{n,d}$, but we will not need to do so.

\section{\label{sec:Sieves}Proofs of the Sieve Lemma and other tools}

In this section we prove several results that we will use as tools
in the sections that follow, including Theorem~\ref{thm:EquivUnivGenBinomDiv},
Lemma~\ref{lem:SgdivSieve}, and Propositions~\ref{pro:UgenAltPP}
and \ref{prop:PrimeGenerationFails}.

\subsection{\label{subsec:ProofEquivUgenBinom}Proof of Theorem~\ref{thm:EquivUnivGenBinomDiv}}

Suppose $n$ satisfies Condition~(\ref{enu:UnivGenCond}) with $p$
and $r$. As described in Section~\ref{sec:Preliminaries}, the maximal
intransitive subgroups of $A_{n}$ are stabilizers of $k$-subsets
of $[n]$, and have index ${n \choose k}$ in $A_{n}$. Hence, $n$
also satisfies Condition~(\ref{enu:BinomCond}) with $p$ and $r$.
See the discussion following (\ref{eq:ConditionsImplications}). 

Thus, in order to prove Theorem~\ref{thm:EquivUnivGenBinomDiv},
it suffices to show that if $n$ satisfies Condition~(\ref{enu:BinomCond})
with $p$ and $r$, then the index of every primitive or imprimitive
maximal subgroup is divisible by at least one of $p$ or $r$. 

For the primitive case, we use the following version of a classic
theorem of permutation group theory due to Jordan. 
\begin{thm}
\emph{(Jordan} \cite{Jordan:1875}\emph{, see also \cite[Section 3.3]{Dixon/Mortimer:1996})}
\emph{\label{thm:JordanPrimitiveA_n}} Let $n\geq9$, and let $H$
be a primitive subgroup of $A_{n}$.

\begin{enumerate}
\item If $p\leq n-3$ is a prime, and $H$ contains a $p$-cycle, then $H=A_{n}$.
\item If $H$ contains the product of two transpositions, then $H=A_{n}$.
\end{enumerate}
\end{thm}

The next lemma follows quickly.
\begin{lem}
\label{lem:PrimitiveDivIndex} Let $p$ be a prime. If $n\geq9$ and
$p\leq n-3$, then $p$ divides the index of every primitive proper
subgroup of $A_{n}$.
\end{lem}

\begin{proof}
If $p$ is odd, then every Sylow $p$-subgroup of $A_{n}$ contains
a $p$-cycle. Similarly, every Sylow $2$-subgroup of $A_{n}$ contains
an element that is the product of two transpositions. In either case,
Theorem~\ref{thm:JordanPrimitiveA_n} gives that no primitive proper
subgroup of $A_{n}$ contains any Sylow $p$-subgroup of $A_{n}$.
\end{proof}
Since Lemma~\ref{lem:PrimitiveDivIndex} only applies when $n\geq9$,
we pause to handle the situation when $n<9$. The only integer less
than $9$ that is not a prime power is $6$, and the equivalence of
Conditions~(\ref{enu:BinomCond}) and (\ref{enu:UnivGenCond}) for
$n=6$ is obtained by direct inspection (see Table~\ref{tab:IndicesInSmallA_n}
below). 

Now assume as above that $n\geq9$ satisfies Condition~(\ref{enu:BinomCond})
with $p$ and $r$. Since $n$ is not a prime power, we see from Lemma~\ref{lem:KummerTheorem}
that $p$ and $r$ must be distinct, hence one must be smaller than
$n-2$. As $n\geq9$, it follows from Lemma~\ref{lem:PrimitiveDivIndex}
that the index of every primitive proper subgroup is divisible by
at least one of $p$ or $r$, as desired.

We now handle the imprimitive case, using Lemma~\ref{lem:ImprimitiveDivIndex}.
Let $d$ be a divisor of $n$. We notice that if $p$ divides ${n \choose d}$,
then adding $n-d$ and $d$ in base $p$ requires a carry (by Lemma~\ref{lem:KummerTheorem}).
It follows immediately from Lemma~\ref{lem:ImprimitiveDivIndex}
that the index $\frac{n!}{(d!)^{n/d}\cdot(n/d)!}$ of an imprimitive
maximal subgroup is divisible by either $p$ or $r$, except possibly
if $d$ is a power of $p$ or $r$. 

Suppose that $d$ is a power of $p$, and that $p^{a}$ is the highest
power of $p$ dividing $n$. Then Lemma~\ref{lem:KummerTheorem}
shows that ${n \choose p^{a}}$ is not divisible by $p$, hence it
is divisible by $r$. Adding $n/p^{a}$ copies of $p^{a}$ in base
$r$ therefore requires a carry. Since $d\leq p^{a}$, adding $n/d$
copies of $d$ in base $r$ will also require a carry. Therefore,
$\frac{n!}{(d!)^{n/d}\cdot(n/d)!}$ is divisible by $r$, as desired.
The case where $d$ is a power of $r$ is handled similarly.

\subsection{\label{subsec:ProofHardSgsieve}Proof of Lemma~\ref{lem:SgdivSieve}~(\ref{enu:SgSieveUgen})}

Kummer's Theorem (Lemma~\ref{lem:KummerTheorem}) gives us the following.
\begin{lem}
\label{lem:SieveIntransitive} Let $n$ be a positive integer, and
let $p$ and $r$ be distinct primes. If there are positive integers
$a$ and $b$ such that $p^{a}\mid n$ and $r^{b}<n<p^{a}+r^{b}$,
then for $0<k<n$ at least one of $p,r$ divides ${n \choose k}$. 
\end{lem}

\begin{proof}
Notice that since $p^{a}>n-r^{b}$, either $k<p^{a}$ or else $k>n-r^{b}$.
We assume without loss of generality that $k\leq n/2$.

Let $k=\sum k_{i}p^{i}$ and $n=\sum n_{i}p^{i}$ respectively be
the base $p$ expansions of $k$ and $n$. As $p^{a}\mid n$, therefore
$n_{i}=0$ for $i<a$.  When $k<p^{a}$, then $k_{j}=0$ for all $j\geq a$.
Since $k\neq0$, there is a carry when adding $k$ and $n-k$ in base
$p$. It follows from Lemma~\ref{lem:KummerTheorem} that $p\mid{n \choose k}$. 

When $k>n-r^{b}$, we notice that $k\leq n/2<r^{b}$, and therefore
both $k$ and $n-k$ are between $n-r^{b}$ and $r^{b}$. In particular,
the $b$th place of the base $r$ expansion of both $k$ and $n-k$
has a 0. Since $n/2<r^{b}<n$, the $b$th place of the base $r$ expansion
of $n$ has a $1$. It follows that there is a carry when adding $k$
and $n-k$, hence by Lemma~\ref{lem:KummerTheorem} that $r\mid{n \choose k}$.
\end{proof}
Lemma~\ref{lem:SgdivSieve}~(\ref{enu:SgSieveUgen}) follows from
Lemma~\ref{lem:SieveIntransitive} and Theorem~\ref{thm:EquivUnivGenBinomDiv}.

\subsection{Proof of Lemma~\ref{lem:SgdivSieve}~(\ref{enu:SgSievePrimeGen})}

Let $x\in A_{n}$ have cycle type $p^{n/p}$, that is, let $x$ be
the product of $n/p$ pairwise disjoint $p$-cycles. (Since $p\neq2$,
a $p$-cycle is an even permutation.) Let $y\in A_{n}$ be an $r$-cycle.
We take $C$ to be the conjugacy class containing $x$, and $D$ to
be the conjugacy class containing $y$. Since we chose $(x,y)$ arbitrarily
from $C\times D$, it is enough to show $\left\langle x,y\right\rangle =A_{n}$,
that is, that $\left\langle x,y\right\rangle $ is not contained in
a maximal subgroup of any of the three types discussed in Section~\ref{sec:Preliminaries}.

Since $r<n-2$, it is immediate from Theorem~\ref{thm:JordanPrimitiveA_n}
that $\left\langle y\right\rangle $ is contained in no maximal primitive
subgroup. 

If $p$ is a proper divisor of $n$, we see that $p\leq n/2$ and
hence that $r>n-p\geq n/2$. It is then immediate by Corollary~\ref{cor:ImprimitiveDivLargepp}
that $\left\langle y\right\rangle $ is contained in no imprimitive
maximal subgroup. Otherwise, if $n=p$, then $A_{n}$ has no imprimitive
maximal subgroups.

It remains to show that $\left\langle x,y\right\rangle $ is transitive
in the natural action on $[n]$. Since $y$ acts transitively on an
$r$-set $Y\subseteq[n]$, it suffices to show that every $i\in[n]$
can be moved into $Y$ by $x$. But $i$ is permuted in a $p$-cycle
by $x$, and since $r+p>n$, some element of this $p$-cycle must
be in $Y$, as desired. 

\subsection{Proof of Proposition~\ref{pro:UgenAltPP}}

Direct inspection verifies the proposition for $n\leq8$. See Table~\ref{tab:IndicesInSmallA_n}
below. We assume henceforth that $n\geq9$.

We first verify part~(B). By the Bertrand-Chebyshev Theorem \cite[Theorem 8.7]{Niven/Zuckerman/Montgomery:1991}
there is a prime $r$ with $n/2<r<n-2$. We let $x$ be any $r$-cycle,
and notice that $\left\langle x\right\rangle $ is a Sylow $r$-subgroup.
Then $r$ divides the index of any imprimitive or primitive maximal
subgroup by Corollary~\ref{cor:ImprimitiveDivLargepp} and Lemma~\ref{lem:PrimitiveDivIndex}
respectively.

We now take $y$ to be any $n$-cycle in the case where $n=p^{a}$
is odd, or the product of any two disjoint $2^{a-1}$-cycles in the
case where $n=2^{a}$ is even. In the former case, $\left\langle y\right\rangle $
is transitive. In the latter case, as $r>2^{a-1}$, we see that $\left\langle x,y\right\rangle $
is transitive. In either case, $\left\langle x,y\right\rangle $ is
contained in no intransitive maximal subgroup, hence $\left\langle x,y\right\rangle =A_{n}$.
Since conjugation fixes cycle type, part~(B) follows.

It remains to verify (A). In the case where $n$ is even, it follows
from part~(B). Otherwise, we take $y$ to be any $n$-cycle. Then
$\left\langle y\right\rangle $ is transitive, while Lemmas~\ref{lem:ImprimitiveDivIndex}
and \ref{lem:PrimitiveDivIndex} give that no imprimitive or primitive
maximal subgroup contains a Sylow $2$-subgroup. It follows that $\left\langle y,P\right\rangle =A_{n}$
for any Sylow $2$-subgroup $P$, completing the proof of part~(A).

\subsection{Proof of Proposition~\ref{prop:PrimeGenerationFails}}

Let $C$ and $D$ be as in Condition~(\ref{enu:PrimeConjGenCond}).
We will find $(c,d)\in C\times D$ such that $\left\langle c,d\right\rangle \neq A_{n}$.

Since $A_{n}$ is transitive, if $D$ does not consist of derangements
then we may find an element $d$ of $D$ fixing $n$. The same holds
for $C$. If $c$ and $d$ both fix $n$, then $\left\langle c,d\right\rangle $
is intransitive, hence a proper subgroup of $A_{n}$. This reduces
us to the situation where one conjugacy class (without loss of generality
$C$) consists of derangements.

Since $n=2^{a}$, derangements of prime order in $A_{n}$ are fixed-point-free
involutions. It is straightforward to verify that the fixed-point-free
involutions of $A_{n}$ form a single conjugacy class. Thus, $C$
consists of all fixed-point-free involutions in $A_{n}$. 

Since a Sylow $2$-subgroup of $A_{n}$ intersects every conjugacy
class of involutions nontrivially, we see that $D$ must consist of
elements of odd prime order $p$. For any $d\in D$, every orbit of
$\left\langle d\right\rangle $ is of size $1$ or $p$. If $\left\langle d\right\rangle $
has more than two orbits, then let $O_{1}$ and $O_{2}$ be orbits.
Now there is some $c\in C$ such that $O_{1}\cup O_{2}$ is the union
of the supports of $2$-cycles in the disjoint cycle decomposition
of $c$. The subgroup $\left\langle c,d\right\rangle $ is thus intransitive.

It remains only to consider the case where $\left\langle d\right\rangle $
has exactly two orbits. As $n=2^{a}$, so $d$ is a $p$-cycle fixing
exactly one point. Now $n=p+1$, and so by the Sylow Theorems the
subgroups of order $p$ in $A_{n}$ form a single conjugacy class.
Thus, it suffices to find a proper subgroup of $A_{n}$ that contains
both a fixed-point-free involution $c$ and an element $d$ of order
$p$.

Consider the transitive action of $PSL_{2}(p)$ on the set $\mathcal{S}$
of $1$-dimensional subspaces of $\mathbb{F}_{p}^{2}$. Since $\left|\mathcal{S}\right|=n$
and $PSL_{2}(p)$ is simple, we obtain from the group action a subgroup
$H\cong PSL_{2}(p)$ of $A_{n}$. Then $\left|H\right|=\frac{p\cdot(p^{2}-1)}{2}$,
and by the Orbit-Stabilizer Theorem, the stabilizer of any point has
order $\frac{p\cdot(p-1)}{2}=p\cdot(2^{a-1}-1)$. In particular, the
subgroup $H$ contains elements of order $p$ and order $2$, and
no element of order $2$ in $H$ fixes any point.
\begin{rem}
Powers of $2$ satisfy Condition~(\ref{enu:PPConjGenCond}) by Proposition~\ref{pro:UgenAltPP}.
\end{rem}

\subsection{\label{subsec:SmallAltGroups}Very small alternating groups}

As Lemma~\ref{lem:SgdivSieve} does not apply when $n\leq8$, we
examine small $n$ separately. The solvable alternating groups (where
$n<5$) all trivially satisfy Condition~(\ref{enu:PrimeConjGenCond}).
For $5\leq n\leq8$, we present in Table~\ref{tab:IndicesInSmallA_n}
the indices of maximal subgroups of $A_{n}$, together with representatives
for generating conjugacy classes as in Condition~(\ref{enu:PPConjGenCond}).
This list is easy to produce either by GAP \cite{GAP4.5.5}, or else
by hand (using well-known facts about primitive groups of small degree). 

For $n=5$ or $7$, these representatives are of prime order, so $5$
and $7$ satisfy Condition~(\ref{enu:PrimeConjGenCond}). Proposition~\ref{prop:PrimeGenerationFails}
tells us that $8$ fails Condition~(\ref{enu:PrimeConjGenCond}),
and a similar argument or GAP computation shows that $6$ also fails
Condition~(\ref{enu:PrimeConjGenCond}).

\begin{table}
\begin{tabular}{c||l||l}
$n$ &
Maximal subgroup indices &
Cond.~(\ref{enu:PPConjGenCond}) conj. class representatives\tabularnewline
\hline 
\hline 
5 &
5, 6, 10 &
$(1\,2\,3),(1\,2\,3\,4\,5)$\tabularnewline
6 &
6, 10, 15 &
$(1\,2\,3\,4)(5\,6),(1\,2\,3\,4\,5)$\tabularnewline
7 &
7, 15, 21, 35 &
$(1\,2\,3\,4\,5),(1\,2\,3\,4\,5\,6\,7)$\tabularnewline
8 &
8, 15, 28, 35, 56 &
$(1\,2\,3\,4)(5\,6\,7\,8),(1\,2\,3\,4\,5)$\tabularnewline
\end{tabular}

\bigskip{}

\caption{\label{tab:IndicesInSmallA_n}Indices of maximal subgroups and generating
conjugacy class representatives for $A_{n}$, $5\leq n\leq8$. }
\end{table}

\section{\label{sec:Asymptotic-density}Asymptotic density}

In this section, we use Part~(\ref{enu:SgSievePrimeGen}) of Lemma~\ref{lem:SgdivSieve}
to prove Theorem~\ref{thm:AsymptoticDensity}. 

Lemma~\ref{lem:SgdivSieve} tells us that $n$ satisfies Condition~(\ref{enu:PrimeConjGenCond})
unless both the largest prime divisor $p$ of $n$ and the largest
prime $r$ that is less than $n-2$ are small relative to $n$. This
allows us to apply known and conjectured results about prime gaps,
which we combine with known results about numbers without large prime
divisors (``smooth numbers''). 

We  will use the following notation. 
\begin{itemize}
\item We will denote the $k$th smallest prime number by $p_{k}$. For example,
$p_{1}=2$ and $p_{2}=3$.
\item For a real number $x>2$, we will denote by $r(x)$ the largest prime
that is no larger than $x$.
\item For positive real numbers $x,y$, we will denote by $\Psi(x,y)$ the
number of positive integers no larger than $x$ which have no prime
factor larger than $y$.
\end{itemize}
Our strategy is to show that if $p$ is the largest prime divisor
of $n$, then asymptotically $r(n)+p$ is frequently greater than
$n$. We remark that $r(n)\geq n-2$ only on a set of asymptotic density
0, so we may treat the $r+2<n$ condition of Lemma~\ref{lem:SgdivSieve}~(\ref{enu:SgSievePrimeGen})
as reading $r\leq n$ for the purpose of asymptotic density arguments.

We will require several tools from number theory, as we will describe
below. See \cite{Granville:2008} for further background on (\ref{eq:DickmanEstimate})
and (\ref{eq:RankinEstimate}), and \cite{Granville:1995} for background
and history on (\ref{eq:CramerConjecture}) and (\ref{eq:CramerRHEstimate}). 

\subsection{Proof of Theorem~\ref{thm:AsymptoticDensity}~(A)}

Jia showed in \cite{Jia:1996} that, for any $\epsilon>0$, there
is a prime on the interval $[n,n+n^{\frac{1}{20}+\epsilon}]$ for
all $n$ excluding a set of asymptotic density $0$. It follows by
routine manipulation that 
\begin{equation}
n-r(n)<n^{\frac{1}{20}}\quad\mbox{except on a set of asymptotic density }0.\label{eq:Jia}
\end{equation}
 See \cite[Chapter 9]{Harman:2007} for further discussion of results
of this type.

Dickman showed in \cite{Dickman:1930} that 
\begin{equation}
\lim_{x\rightarrow\infty}\frac{\Psi(x,x^{1/u})}{x}=\rho(u)\mbox{\quad for any fixed }u,\label{eq:DickmanEstimate}
\end{equation}
where $\rho$ denotes the so-called \emph{Dickman-de Bruijn function},
that is, the solution to the differential equation $u\rho'(u)+\rho(u-1)=0$. 

By combining (\ref{eq:Jia}) and (\ref{eq:DickmanEstimate}) with
Lemma~\ref{lem:SgdivSieve}~(\ref{enu:SgSievePrimeGen}), we see
that the desired asymptotic density $\alpha$ satisfies 
\[
\alpha\geq1-\rho(20),
\]
as desired. Consulting the table of values for $\rho$ in \cite[Table 2]{Granville:2008},
we see that $\rho(20)\cong2.462\cdot10^{-29}<10^{-28}$.

\subsection{Proof of Theorem~\ref{thm:AsymptoticDensity}~(B)}

Rankin showed in \cite{Rankin:1938} that 
\begin{equation}
\lim_{x\rightarrow\infty}\frac{\Psi(x,\log^{b}x)}{x}=0\mbox{\quad for any }b>1.\label{eq:RankinEstimate}
\end{equation}
Taking $b=3$ in (\ref{eq:RankinEstimate}), we see that the set of
integers $n$ with no prime factor larger than $\log^{3}n$ has asymptotic
density 0. 

The Cramér Conjecture \cite[(4)]{Cramer:1936} says that there is
a constant $C$ such that 
\begin{equation}
p_{k+1}-p_{k}\leq C\log^{2}p_{k}\quad\mbox{for all }k.\label{eq:CramerConjecture}
\end{equation}
In the same paper, Cramér \cite[Theorem II]{Cramer:1936} showed the
Riemann Hypothesis to imply that 
\begin{equation}
\lim_{x\rightarrow\infty}\,\frac{1}{x}\,\,\,\,\,\cdot\negmedspace\negmedspace\negmedspace\negmedspace\negmedspace\sum_{\substack{p_{k}\leq x,\\
p_{k+1}-p_{k}\geq\log^{3}p_{k}
}
}(p_{k+1}-p_{k})=0.\label{eq:CramerRHEstimate}
\end{equation}
Thus, if either the Cramér Conjecture or the Riemann Hypothesis holds,
then 
\begin{equation}
n-r(n)\leq\log^{3}r(n)\leq\log^{3}n\label{eq:LogCubedEstimate}
\end{equation}
except on a set of asymptotic density zero. Theorem~\ref{thm:AsymptoticDensity}~(B)
follows upon combining (\ref{eq:RankinEstimate}) with $b=3$, (\ref{eq:LogCubedEstimate}),
and Lemma~\ref{lem:SgdivSieve}.

\section{\label{sec:Computation}Computational results}

\begin{table}
\begin{tabular}{rclc|c|lc}
$n$ &
 &
 &
 &
$p^{a}$ &
Cond.~(\ref{enu:UnivGenCond}) prime pairs &
\tabularnewline[\doublerulesep]
\hline 
\noalign{\vskip\doublerulesep}
$31\comma416$ &
= &
$2^{3}\cdot3\cdot7\cdot11\cdot17$ &
 &
$17^{1}$ &
$(2,7853)$ &
\tabularnewline
$46\comma800$ &
= &
$2^{4}\cdot3^{2}\cdot5^{2}\cdot13$ &
 &
$5^{2}$ &
$(2,149)$ &
\tabularnewline
$195\comma624$ &
= &
$2^{3}\cdot3^{2}\cdot11\cdot13\cdot19$ &
 &
$19^{1}$ &
$(2,3)$ &
\tabularnewline
$5\comma504\comma490$ &
= &
$2\cdot3^{3}\cdot5\cdot19\cdot29\cdot37$ &
 &
$37^{1}$ &
$(3,5)$ &
\tabularnewline
$7\comma458\comma780$ &
= &
$2^{2}\cdot3\cdot5\cdot7^{2}\cdot43\cdot59$ &
 &
$59^{1}$ &
$(2,276251)$ &
\tabularnewline
$9\comma968\comma112$ &
= &
$2^{4}\cdot3^{2}\cdot7\cdot11\cdot29\cdot31$ &
 &
$31^{1}$ &
$(2,3)$ &
\tabularnewline
$12\comma387\comma600$ &
= &
$2^{4}\cdot3^{3}\cdot5^{2}\cdot31\cdot37$ &
 &
$37^{1}$ &
$(2,3)$ &
\tabularnewline
$105\comma666\comma600$ &
= &
$2^{3}\cdot3\cdot5^{2}\cdot13\cdot19\cdot23\cdot31$ &
 &
$31^{1}$ &
$(2,5)$ &
\tabularnewline
$115\comma690\comma848$ &
= &
$2^{5}\cdot3\cdot7\cdot13\cdot17\cdot19\cdot41$ &
 &
$41^{1}$ &
$(2,3)$ &
\tabularnewline
$130\comma559\comma352$ &
= &
$2^{3}\cdot3\cdot7\cdot11\cdot31\cdot43\cdot53$ &
 &
$53^{1}$ &
$(2,112843)$ &
\tabularnewline
$146\comma187\comma444$ &
= &
$2^{2}\cdot3\cdot13\cdot19\cdot31\cdot37\cdot43$ &
 &
$43^{1}$ &
$(2,31)$ &
\tabularnewline
$225\comma613\comma050$ &
= &
$2\cdot3\cdot5^{2}\cdot13\cdot37\cdot53\cdot59$ &
 &
$59^{1}$ &
$(2,516277)$ &
\tabularnewline
$275\comma172\comma996$ &
= &
$2^{2}\cdot3\cdot7\cdot29\cdot37\cdot43\cdot71$ &
 &
$71^{1}$ &
$(2,567367)$ &
\tabularnewline
$282\comma429\comma840$ &
= &
$2^{4}\cdot3\cdot5\cdot7\cdot11\cdot17\cdot29\cdot31$ &
 &
$31^{1}$ &
$(2,29)$ &
\tabularnewline
$300\comma688\comma752$ &
= &
$2^{4}\cdot3\cdot7\cdot13\cdot23\cdot41\cdot73$ &
 &
$73^{1}$ &
$(2,11)$ &
\tabularnewline
$539\comma509\comma620$ &
= &
$2^{2}\cdot3\cdot5\cdot13\cdot17\cdot23\cdot29\cdot61$ &
 &
$61^{1}$ &
$(2,1201)$ &
\tabularnewline
$653\comma426\comma796$ &
= &
$2^{2}\cdot3\cdot11\cdot19\cdot43\cdot73\cdot83$ &
 &
$83^{1}$ &
$(2,73)$ &
\tabularnewline
$696\comma595\comma536$ &
= &
$2^{4}\cdot3^{2}\cdot7\cdot13\cdot17\cdot53\cdot59$ &
 &
$59^{1}$ &
$(2,13)$ &
\tabularnewline
$784\comma474\comma592$ &
= &
$2^{5}\cdot11\cdot29\cdot31\cdot37\cdot67$ &
 &
$67^{1}$ &
$(2,29)$ &
\tabularnewline
$798\comma772\comma578$ &
= &
$2\cdot3\cdot19\cdot29\cdot41\cdot71\cdot83$ &
 &
$83^{1}$ &
$(2,563)$ &
\tabularnewline
$815\comma224\comma800$ &
= &
$2^{5}\cdot3\cdot5^{2}\cdot13\cdot17\cdot29\cdot53$ &
 &
$53^{1}$ &
$(2,87013)$ &
\tabularnewline
$851\comma716\comma320$ &
= &
$2^{5}\cdot3\cdot5\cdot7\cdot13\cdot17\cdot31\cdot37$ &
 &
$37^{1}$ &
$(2,31)$ &
\tabularnewline
\end{tabular}\bigskip{}
\caption{\label{tab:ExceptionsToLargePP}The values of $n\leq1\protect\comma000\protect\comma000\protect\comma000$
together with their maximal prime power divisors $p^{a}$, such that
$n$ does not satisfy Condition~(\ref{enu:UnivGenCond}) with $p$.
Each such $n$ satisfies Condition~(\ref{enu:UnivGenCond}) with
either $2$ or $3$.}
\end{table}

In this section we describe the verification by computer of Proposition~\ref{prop:MainComputation}.

Our program iterates through the integers, beginning with $n=9$.
We factor each integer into primes. If $n$ is a prime power, then
$n$ satisfies Condition~(\ref{enu:PPOverSylConjGenCond}) and hence
(\ref{enu:UnivGenCond}) by Proposition~\ref{pro:UgenAltPP}. In
this case, we store $n=r^{b}$ as the largest prime power known so
far in the computation. Otherwise, we find the largest prime power
$p^{a}$ dividing $n$. The program then checks whether $p^{a}+r^{b}$
is greater than $n$, where $r^{b}$ is the largest prime power found
so far. If so, then $n$ satisfies Condition~(\ref{enu:UnivGenCond})
with $p$ and $r$ by Lemma~\ref{lem:SgdivSieve}. This sieving method
succeeds for all  but $14\comma638$ of the integers in the interval
from $9$ to $1\comma000\comma000\comma000$. For these remaining
integers, the program checks directly which indices of intransitive
and imprimitive subgroups are divisible by $p$ (using Lemmas~\ref{lem:KummerTheorem}
and \ref{lem:ImprimitiveDivIndex}), and searches for a prime $r$
dividing those that are not. This second method works for all but
$22$ of the remaining $14\comma638$ integers. For these $22$ integers
we perform a similar search, using divisors of $n$ other than $p$.
See Table~\ref{tab:ExceptionsToLargePP} for the results of this
search.

Running this program out to $n=1\comma000\comma000\comma000$ on a
2012 MacBook Pro with the GAP computer algebra system \cite{GAP4.5.5}
takes around 2 weeks. This computation verifies Proposition~\ref{prop:MainComputation}.

\smallskip{}

\begin{figure}
\includegraphics[bb=0bp 0bp 479bp 370bp,scale=0.7]{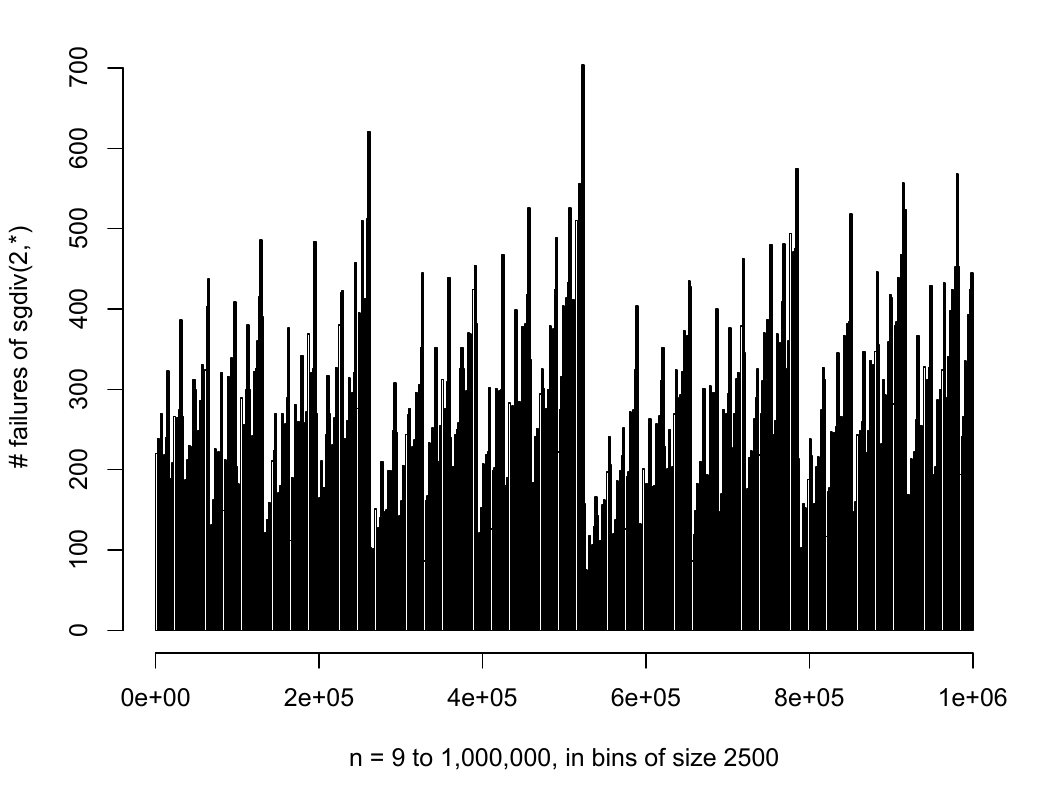}\vspace{-0.5cm}
\caption{\label{fig:DensitySgdiv2Hist}A histogram showing the density of integers
that do not meet Condition~(\ref{enu:UnivGenCond}) with the prime
$2$.}
\end{figure}

We approach checking which values of $n$ satisfy Condition~(\ref{enu:UnivGenCond})
with the prime $2$ in a similar fashion. When we apply Lemma~\ref{lem:SgdivSieve},
we look for a pair $p^{a}+r^{b}>n>r^{b}$ (where $p^{a}\mid n$) as
before, but now we require $2\in\{p,r\}$. This technique gives a
positive answer for about 45.7\% of the first $1\comma000\comma000$
integers $n\geq9$. The remaining values of $n$ require significantly
more computation, and as a result we did not examine values of $n$
beyond $1\comma000\comma000$. 

Running the program to check Condition~(\ref{enu:UnivGenCond}) with
the prime $2$ out to $n=1\comma000\comma000$ takes around a day
on a 2012 MacBook Pro. This computation verifies Proposition~\ref{prop:2Computation}.
More precisely, $867\comma247$ of the integers between $9$ and $1\comma000\comma0000$
satisfy Condition~(\ref{enu:UnivGenCond}) with $2$. The histogram
in Figure~\ref{fig:DensitySgdiv2Hist} shows the density of those
$n$ which do not satisfy Condition~(\ref{enu:UnivGenCond}) with
$2$. We remark that this histogram appears to show that the failing
values are concentrated towards the values of $n$ slightly preceding
integers that are divisible by a high power of $2$.

Source code and output for all computer programs discussed in this
section is available through the arXiv as ancillary files. They are
also currently available from the second author's web page. A list
of the values of $n\leq1\comma000\comma000$ such that $n$ does not
satisfy Condition~(\ref{enu:UnivGenCond}) with the prime $2$ can
be found in the same places.

\bibliographystyle{2_Users_russw_Documents_Research_mypapers_Divisibility_of_binomial_coefficients_hamsplain}
\bibliography{1_Users_russw_Documents_Research_mypapers_Divisibility_of_binomial_coefficients_Master}

\end{document}